\documentclass[12pt]{amsart}

\usepackage{amsmath,amsthm,amssymb}

\usepackage{fancyvrb }

\usepackage[margin=1in]{geometry}

\usepackage[hidelinks]{hyperref}

\usepackage{verbatim}

\usepackage{cases}

\usepackage{color}

\usepackage{todonotes}

\newenvironment{enumalph}{\begin{enumerate}  }{\end{enumerate}}

\newtheorem{theorem}{Theorem}[section]

\newtheorem{proposition}[theorem]{Proposition}

\newtheorem{question}[theorem]{Question}

\newtheorem{corollary}[theorem]{Corollary}

\newtheorem{lemma}[theorem]{Lemma}
\input xy
\xyoption{all}

\theoremstyle{definition}

\newtheorem{remark}[theorem]{Remark}

\DeclareMathOperator{\rank}{rank}

\DeclareMathOperator{\chr}{char}

\DeclareMathOperator{\Ext}{Ext}

\DeclareMathOperator{\opH}{H}
\DeclareMathOperator{\Hom}{Hom}

\DeclareMathOperator{\ind}{ind}

\newcommand{\fraku}{\mathfrak{u}}

\begin{document}

\title[Cohomology of $SL_2$ and related structures]{Cohomology of $SL_2$ and related structures}

\author{Klaus Lux}
\address{Department of Mathematics\\ University of Arizona \\ Tucson\\ AZ~85721, USA}
\email{klux@math.arizona.edu}

\author{Nham V. Ngo}
\address{Department of Mathematics\\ University of Arizona \\ Tucson\\ AZ~85721, USA}
\email{nhamngo@math.arizona.edu}

\author{Yichao Zhang}
\address{Department of Mathematics\\ University of Arizona \\ Tucson\\ AZ~85721, USA}
\email{zhangyichao2002@gmail.com}
\email{yichaozhang@math.arizona.edu}

\date{\today}

\maketitle

\begin{abstract}
Let $SL_2$ be the rank one simple algebraic group defined over an algebraically closed field $k$ of characteristic $p>0$. The paper presents a new method for computing the dimension of the cohomology spaces $\opH^n(SL_2,V(m))$ for Weyl $SL_2$-modules $V(m)$. We provide a closed formula for $\dim\opH^n(SL_2,V(m))$ when $n\le 2p-3$ and show that this dimension is bounded by the $(n+1)$-th Fibonacci number. This formula is then used to compute $\dim\opH^n(SL_2, V(m))$ for $n=1, 2,$ or $3$. For $n>2p-3$, an exponential bound, only depending on $n$, is obtained for $\dim\opH^n(SL_2,V(m))$. Analogous results are also established for the extension spaces $\Ext^n_{SL_2}(V(m_2),V(m_1))$ between Weyl modules $V(m_1)$ and $V(m_2)$. In particular, we determine the degree three extensions for all Weyl modules of $SL_2$. As a byproduct, our results and techniques give explicit upper bounds for the dimensions of the cohomology of the Specht modules of symmetric groups, the cohomology of simple modules of $SL_2$, and the finite group of Lie type $SL_2(p^s)$. 
\end{abstract}

\section{Introduction}

The problem of computing the cohomology of $SL_2$ with coefficients in various $SL_2$-modules is challenging despite the fact that representation theory of $SL_2$ has been extensively studied. Even in the case of simple modules, only partial results can be found in the literature, see \cite[14.7]{Hum} for a short survey on the topic. On the other hand, in the case of Weyl modules, a little more is known. For example, the first degree extension spaces between Weyl modules were first computed by Erdmann in 1995 \cite{E:1995}. Five years later, Cox and Erdmann calculated the second degree extension spaces \cite{CE}. In 2007, Parker \cite{Pa:2006} introduced a recursive formula for computing the extension spaces between Weyl modules. Recently, Erdmann, Hannabuss, and Parker used results in \cite{Pa:2006} and techniques in generating functions to show that the sequence $$\left(\max_{m\in X^+}\{\dim\opH^n(SL_2,V(m))\}\right)_{n\in\mathbb{N}},$$ with $X^+$ the set of dominant weights (see Section \S\ref{notation}), grows at least exponentially in $n$ \cite{EHP:2013}. Preferably, one would like to give a closed formula for the dimension of these extension spaces. This task is the original motivation for the paper.

To compute $\opH^n(SL_2,V(m))$, the common strategy used in the aforementioned papers is a combination of computing the cohomology of the Weyl modules $V(m)$ over the first Frobenius kernel $G_1$ and applying the Lyndon-Hochschild-Serre spectral sequence to $G_1$ and $G/G_1$. We introduce in this paper a new method, which uses information of all Frobenius kernels of $SL_2$ and then employing a stability theorem of Cline, Parshall, and Scott \cite[Theorem 7.4]{CPS}. This new approach allows us to equate the dimension of $\opH^n(SL_2,V(m))$ to the number of solutions for a certain linear system. By studying the latter, we obtain new and interesting results for the dimension of the cohomology and extensions of Weyl modules. 


The paper is organized as follows. After setting up necessary notation and background in Section \ref{notation}, we study the number of solutions $N(m,n)$ for a certain type of linear system in Section \ref{counting}. In particular, for $n\le 2p-2$, we determine the exact value of $N(m,n)$ for all $m$, cf. Lemma \ref{key_lemma}, and show that this number is bounded from above by the $n$-th Fibonacci number $F(n)$, cf. Proposition \ref{Fibonacci}. For $n>2p-2$, we obtain a recursive formula for $N(m,n)$, cf. Proposition \ref{recursive 1 N(m,n)} or Corollary \ref{recursive 2 N(m,n)}, and an exponential upper bound $C_n$ only depending on $n$, cf. Corollary \ref{exponential bound for V(m)}. The bound is established from our investigations on some partition functions. 

Section \ref{cohomology} virtually contains two parts. We first show that the dimension of $\opH^{n}(SL_2,V(m))$ is given by $N(m+1,n+1)$ using mainly the ingredients from the second author's paper \cite{Ngo:2013}. Then we derive results on the cohomology dimension of Weyl modules from the calculations on $N(m,n)$ in the previous section, cf. Theorem \ref{cohomology in V(m)}. We further compute the dimensions of $\opH^n(SL_2,V(m))$ with $n\le 3$, cf. Proposition \ref{low degree cohomology of V(m)}. In addition, when $p=2$, our formula of $\dim\opH^n(SL_2,V(m))$ is exactly the same as that in \cite[Corollary 3.2.2]{EHP:2013} (see Remark \ref{p=2}) and gives a nice upper bound for $\dim\opH^n(SL_2,V(m))$, cf. Theorem \ref{bound for p=2}. Finally, the exponential bound $C_{n+1}$, described in Remark \ref{modification of (b_i)}, of $\dim\opH^n(SL_2,V(m))$ shows that the sequence $\left(\max\{\dim\opH^n(SL_2,V(m)): m\in\mathbb{N}\}\right)$ grows exponentially. This strengthens the main result of Erdmann, Hannabuss, and Parker \cite[Corollary 5.6.3]{EHP:2013}.  

We generalize our work to higher extension spaces between Weyl modules in Section \ref{extension}. The main tools here are the recursive formula of Parker \cite[Theorem 5.1]{Pa:2006} and the low degree cohomology of $V(m)$ computed in the preceding section. Explicitly, we give a complete description for the dimension of the third degree extension space between Weyl modules, and show that the dimension of these spaces is at most 3, cf. Theorem \ref{Ext^3} and Corollary \ref{Ext^3 bound}. This result extends the work of Erdmann et. al. in \cite{E:1995} and \cite{CE}. At the end of the section, we obtain analogs of the bounds for $\dim\opH^n(SL_2,V(m))$, that is, for integers $m_1\ge 0$ and $0\le m_2< p^r$ for some $r\ge 1$, we have for each $n\le 2p-3$,
\[ \dim\Ext^n_{SL_2}(V(m_2),V(m_1))\le F(n+1)+(r-1)F(n), \]
and for arbitrary $n\ge 1$, 
\[ \dim\Ext^n_{SL_2}(V(m_2),V(m_1))\le C_{n+2}+(r-1)C_n. \]

The final section of this paper contains applications of our results and techniques in several topics of group cohomology. We focus on finding a universal bound for the cohomology dimension of various groups, which depends only on the degree. Firstly, using a result of Kleshchev and Nakano on the cohomology of general linear and symmetric groups \cite{KN:2001}, we compute an upper bound of the cohomology dimension of the Specht modules for degree up to $2p-4$  in certain cases, cf. Theorem \ref{Specht module}. Next, we determine universal bounds for the dimension of cohomology of $SL_2$ and finite groups of Lie type $SL_2(p^s)$ with coefficients in simple modules. More precisely, using the combinatorial description of Carlson \cite{Ca:1983} and a computational technique described in Section \ref{counting}, we prove that
\[\dim\opH^n(SL_2(p^s), L)\le (2n+7)C_n,\]
for each simple $SL_2(p^s)$-module $L$. This result is then used together with a theorem on generic cohomology from \cite{CPSvdK} and a recursive formula from \cite{Pa:2006} to derive 
\[\dim\opH^n(SL_2, L(m))\le (2n+7)C_n,\]
for any simple $SL_2$-module $L(m)$. The proof is rather interesting as it indicates that results for cohomology of algebraic groups can be obtained from that of finite groups of Lie type, which is opposite to the usual trend. It is worth noting that the existence of such universal bounds was originally proved for simple algebraic groups by Parshall and Scott \cite{PS:2011} and recently extended to finite groups of Lie type in \cite{BNPPSS}. However, there are no explicit bounds in the literature. Last but not least, these universal bounds  show that the growth of the cohomology dimension for both groups is exponential, which significantly extends the main result of Stewart in \cite{Ste}.              

\section{Some preliminaries}\label{notation}

In the following, we consider $SL_2$ as a simple algebraic group defined over an algebraically closed field $k$ with $\chr(k)=p>0$, and we denote $G=SL_2$. We use standard notation as, for example, described in Jantzen's book \cite{Jan:2003}. For the convenience of the reader, we review a few important terminologies as follows. 

Let $B$ be a Borel subgroup of $G$ and $T$ a maximal torus in $B$. Let $U$ be the unipotent radical of $B$. We will write $X(T)$ for the weight lattice of $T$, and let $\varpi$ be the fundamental weight in $X(T)$. It then follows that $X(T)=\mathbb{Z}\varpi$. Abbreviating $n\varpi$ by $n$ for each $n\in\mathbb{Z}$, we identify $X(T)$ with $\mathbb{Z}$. Next, we denote by $X^+$ the set of dominant weights in $X(T)$ and so we have $X^+=\mathbb{N}$. Note that we consider $\mathbb{N}$ to be the set of non-negative integers. Let
\[ X^+_r=\{ m\varpi\in X^+: m< p^r\}.\]
When $r=1$, $X^+_1$ is called the set of restricted dominant weights. The root system of $G$ is denoted by $\Phi=\{\pm\alpha\}$ with $\alpha=2\varpi$, so that the root lattice $\mathbb{Z}\Phi=2\mathbb{Z}$. 

For a given positive integer $r$, let $F_r:G\to G^{(r)}$ be the $r$-th Frobenius morphism, see for example \cite[I.9]{Jan:2003}. The scheme-theoretic kernel $G_r=\ker(F_r)$ is called the $r$-th Frobenius kernel of $G$. Given a closed subgroup $H$ of $G$, we write $H_r$ for the scheme-theoretic kernel of the restriction $F_r:H\to H^{(r)}$. In other words, we have
\[ H_r=H\cap G_r. \] 
Given a $G$-module $M$, we write $M^{(r)}$ for the module obtained by twisting the structure map for $M$ by $F_r$. Note that $G_r$ acts trivially on $M^{(r)}$. Conversely, if $N$ is a $G$-module on which $G_r$ acts trivially, then there is a unique $G$-module $M$ with $N = M^{(r)}$. We denote the module $M$ by $N^{(-r)}$. Throughout the paper, tensor products will be taken over the field $k$ unless otherwise stated. 

Let $V$ be a $B$-module. Then the induced $G$-module $\ind_B^G(V)$ is defined as 
\[ \ind_B^G(V)=(k[G]\otimes V)^B. \]
The higher derived functors of $\ind_B^G(-)$ are denoted by $R^i\ind_B^G(-)$, but we will only write $\opH^i(-)$ for brevity. Recall from \cite[Proposition 4.5]{Jan:2003} that, for all $\lambda\in X^+$, $\opH^i(\lambda)=0$ for all $i\ge 1$. Now we denote for each $G$-module $M$ the dual $G$-module $\Hom(M,k)$ by $M^*$. Then for each $\lambda\in X^+$, the Weyl module (of highest weight $\lambda$) is defined as $V(\lambda):=\ind_B^G(\lambda)^*$.

For each $G$-module $M$, we define $\Ext^*_G(M,-)$ to be the derived functor of $\Hom_G(M,-)$, see \cite[I.4.2]{Jan:2003} for details. Suppose $N$ is also a $G$-module. Then the $n$-th degree cohomology space of $G$ with coefficients in $N$ is defined to be
\[ \opH^n(G,N)=\Ext^n_G(k,N). \]

We denote $\mathbb{Z}^+$ for the set of positive integers. Given $m\in\mathbb{Z}^+$, let $r_m$ be the largest positive integer such that $p^{r_m}\le m$ and let
\[c_0+c_1p+\cdots+c_{r_m}p^{r_m}\]
be the $p$-adic expansion of $m$ with $c_i\in\{0,\ldots,p-1\}$. The height of $m$, denoted by $\textrm{ht}(m)$, is defined by $\sum_{i=0}^{r_m}c_i$. We denote further by $s_m$ the number of zeros in the sequence $\{c_1,\ldots,c_{r_m}\}$. Hence, $r_m-s_m$ is the number of non-zero $c_i$'s with $1\le i\le r_m$.

For $x\in \mathbb{R}$, we let $\lfloor x\rfloor$ be the floor of $x$, namely the largest integer that is smaller than or equal to $x$. We recall that the Fibonacci sequence is defined by
\[F(1)=1,\quad F(2)=1, \quad  F(n)=F(n-1)+F(n-2), \]
for $n\ge 2$. We define a sequence to grow at least exponentially if there is an exponential lower bound. We then say a sequence grows exponentially if it is bounded above and below by exponential functions. 

\section{Counting the number of solutions of certain linear systems}\label{counting}
In this section, we study the number of solutions of a linear system which will be mainly used in latter parts of the paper. Throughout this section, we assume $p$ is an odd prime, unless otherwise stated.
\subsection{} Given $m, n\in\mathbb{N}$ and $r\in\mathbb{Z}^+$, we are interested in the set of pairs $(\mathfrak a,\mathfrak b)$ with $\mathfrak a=(a_1,\ldots,a_r)\in\mathbb{N}^r$ and $\mathfrak b=(b_1\ldots,b_r)\in\{0,1\}^r$ satisfying the system of two linear equations
\begin{align}\label{system}
\begin{cases}
\displaystyle{2\sum_{i=1}^r a_i+\sum_{j=1}^rb_j=n},\\
\displaystyle{b_1+\sum_{i=1}^{r-1}(a_i+b_{i+1})p^i+a_rp^r=m.}
\end{cases}
\end{align}

For convenience, we denote by $S_r(m,n)$ the set of solutions $(\mathfrak a,\mathfrak b)$ to \eqref{system} and $N_r(m,n)=|S_r(m,n)|$. The following lemma gives a lower bound for $N_r(m,n)$. 

\begin{lemma}\label{key_lemma}
Let $m, n$ be non-negative integers and let 
\[ c_0+c_1p+\cdots+c_{r_m}p^{r_m}\]
be the $p$-adic expansion of $m$ with $0\le c_i\le p-1$. Let $s_m$ be the number of zeros in the sequence $\left(c_1,\ldots,c_{r_m}\right)$. Then for $r\ge r_m+1$, we have 
\[ N_r(m,n)\ge \binom{r_m-s_m}{2\textrm{ht}(m)-c_0-n}. \] 
In particular, the equality holds when $n\le 2p-2$.
\end{lemma}
\begin{proof}
We present one way to get solutions of the system \eqref{system}. Then in the case when $n\le 2p-2$ we show that this method gives all solutions of \eqref{system}. By comparing the $p$-adic expansion of $m$ with the left-hand side of the second equation in \eqref{system}, we have
\begin{align}\label{solutions of the first eqn}
\begin{cases}
b_1 =c_0, & \\
a_i+b_{i+1} =c_i &~~\text{for}~~ 1\le i\le r_m, \\
a_j =b_{j+1}=0 &~~\text{for}~~j\ge r_m+1. 
\end{cases}
\end{align}
Eliminating the $a_i$'s in the first equation of \eqref{system}, we obtain 
\[ 2\sum_{i=0}^{r_m}c_i-\sum_{j=1}^{r_m+1}b_j=n\]
or 
\begin{align}\label{equiv of eqn 1}
 \sum_{j=2}^{r_m+1}b_j=2\textrm{ht}(m)-c_0-n.
\end{align}
Now note that the left-hand side is at most $r_m-s_m$. This shows that the number of ways to choose $b_j$'s is
\[ \displaystyle{\binom{r_m-s_m}{2\textrm{ht}(m)-c_0-n}}.\]
Thus, we have proved the first statement of the lemma. 

For $n\le 2p-2$, it suffices to show that all solutions of the system \eqref{system} are those of two equations \eqref{solutions of the first eqn} and \eqref{equiv of eqn 1}. Indeed, if $n\le 2p-2$ then from the first equation we obtain that $b_1<p,$ and $a_i+b_{i+1}<p$ for all $1\le i\le r_m$. It then follows that the left-hand side of the second equation must be the $p$-adic expansion of $m$. Hence, the system \eqref{system} is equivalent to the system of equations \eqref{solutions of the first eqn} and  \eqref{equiv of eqn 1}. So we are done.
\end{proof}

\begin{remark}
The inequality for the number of solutions of the system \eqref{system} in general is strict. In other words, there are solutions of \eqref{system} that do not satisfy \eqref{solutions of the first eqn}. For example, consider $m=138, n=8$ and $p=3$, then the set
\[ b_1=0, b_2=b_3=b_4=b_5=1, a_1=a_3=a_4=a_5=0, a_2=2 \]
is one solution of \eqref{system} with $a_2+b_3=3=p$ which does not satisfies \eqref{solutions of the first eqn}.
\end{remark}

\subsection{Recursive formula for $N_r(m,n)$}
Recall that Lemma \ref{key_lemma} gives the exact values of $N_r(m,n)$ for small values of $n$, i.e., $n\le 2p-2$. In the following, we compute this $N_r(m,n)$ for all $n$ recursively. 

Note that we may embed $S_r(m,n)$ into $S_{r+1}(m,n)$ via the below map
\begin{align*}
i_0: S_r(m,n) &\to S_{r+1}(m,n),\\
\left((a_1,\ldots,a_r), (b_1,\ldots,b_r)\right) &\mapsto \left((a_1,\ldots,a_r,0), (b_1,\ldots,b_r,0)\right).
\end{align*}
Moreover if we take $r$ such that $p^r>m$, then the set $S_r(m,n)$, hence $N_r(m,n)$, stabilizes, i.e., $i_0(S_{r'}(m,n))=S_{r'+1}(m,n)$ for all $r'\ge r$. Henceforth, until the end of this section, we omit the subscript $r$ in the notation and write $S(m,n)$ and $N(m,n)$. In particular, for $r>\log_p(m)$, $N(m,n)=N_{r}(m,n)$.

\begin{lemma}\label{properties of N(m,n,p)}
For all $m, n\ge 0$, we have
\begin{enumalph}
\item $N(m,n)\ne 0$ only if $m\equiv 0$ or $1$ (mod $p$). 
\item If $p\mid m$ then $N(m,n)=N(m+1,n+1)$.
\item If $p\mid m$ then $N(m,n)=
\begin{cases}
0~~&\text{if}~~ n>\frac{2m}{p},\\
1~~&\text{if}~~ n=\frac{2m}{p}.
\end{cases}
$
\end{enumalph}
\end{lemma} 
\begin{proof}
Part (a) follows since $b_1$ only takes $0$ or $1$.

(b) Note that $N(m,n)$ is the number of solutions for the system \eqref{system} with $b_1=0$ which is $N(m+1,n+1)$ when setting $b_1=1$. Hence, we obtain the equality.    

(c) If $n=\frac{2m}{p}$, then $m=a_1p$ with $a_1=\frac{m}{p}$, showing that $N(m,n)=1$. For $n>\frac{2m}{p}$, there are no solutions for the system \eqref{system}. Hence, $N(m,n)=0$ as stated.
\end{proof}
Next, we provide a recursive formula for $N(m,n)$.

\begin{proposition}\label{recursive 1 N(m,n)}
For non-negative integers $m$ and $n$, we have
\begin{align}\label{recursive formula}
 N(m,n)=
\begin{cases}
\sum\limits_{0\le a\le\lfloor\frac{n}{2}\rfloor}N\left(\frac{m}{p}-a,n-2a\right) ~~~&\text{if}~~~p\mid m,\\
\sum\limits_{0\le a\le\lfloor\frac{n}{2}\rfloor}N\left(\frac{m-1}{p}-a,n-2a-1\right) ~~~&\text{if}~~~m\equiv 1~ (mod~ p),\\
0 &\text{otherwise}.
\end{cases}
\end{align}
\end{proposition}

\begin{proof}
By Lemma \ref{properties of N(m,n,p)}(a), we may assume that $m\equiv \epsilon$ (mod $p$) with $\epsilon=0$ or $1$, i.e., $m=pu+\epsilon$ for some $u$ in $\mathbb{N}$. From the second equation of system \eqref{system} we must have $b_1=\epsilon$. So the system \eqref{system} can be rewritten as
\[ 
\begin{cases}
\sum_{i=1}^r2a_i+\sum_{j=2}^rb_j=n-\epsilon,\\
a_1+b_2+(a_2+b_3)p+\cdots+(a_r)p^{r-1}=u,
\end{cases}
\]
which we rewrite as 
\[ 
\begin{cases}
\sum_{i=2}^r2a_i+\sum_{j=2}^rb_j =n-\epsilon-2a_1,\\
b_2+(a_2+b_3)p+\cdots+(a_r)p^{r-1} =u-a_1.
\end{cases}
\]
By shifting the indices of the $a_i$'s and $b_j$'s, we obtain the number of solutions of the system above is $N(u-a_1,n-\epsilon-2a_1)$. Since $a_1\le\lfloor\frac{n}{2}\rfloor$, we obtain the recursive formula.
\end{proof} 

Indeed, we can reformulate \eqref{recursive formula} in a much nicer form.

\begin{corollary}\label{recursive 2 N(m,n)}
For non-negative integers $m, n$, we have
\[ N(m,n)=
\begin{cases}
N(m-p,n-2)+N(\frac{m}{p},n)~~~&\text{if}~~~p\mid m,\\
N(m-p,n-2)+N(\frac{m-1}{p},n-1)~~~&\text{if}~~~m\equiv 1~(\text{mod}~p).
\end{cases}
\]
\end{corollary}

\begin{proof}
Suppose that $p\mid m$. By \eqref{recursive formula} we get the following sequence of equalities
\begin{align*}
N(m,n) &=\sum\limits_{0\le a\le\lfloor\frac{n}{2}\rfloor}N\left(\frac{m}{p}-a,n-2a\right)\\
&=N\left(\frac{m}{p},n\right)+\sum\limits_{0<a\le\lfloor\frac{n}{2}\rfloor}N\left(\frac{m}{p}-a,n-2a\right)\\
&=N\left(\frac{m}{p},n\right)+\sum\limits_{0\le a'\le\lfloor\frac{n}{2}\rfloor-1}N\left(\frac{m}{p}-(1+a'),n-2(1+a')\right)\\
&=N\left(\frac{m}{p},n\right)+\sum\limits_{0\le a'\le\lfloor\frac{n-2}{2}\rfloor}N\left(\frac{m-p}{p}-a',n-2-2a'\right)\\
&=N\left(\frac{m}{p},n\right)+N(m-p,n-2).
\end{align*}
This proves the first case. Using this result and Lemma \ref{properties of N(m,n,p)}(b), we obtain the other case.
\end{proof}

We now compute a uniform upper bound of $N(m,n)$ for all $m\ge 0$ and fixed $1\le n\le 2p-2$.

\begin{proposition}\label{Fibonacci}
For all $1\le n\le 2p-2$, we have \[N(m,n)\le F(n),\] where $F(n)$ is the $n$-th number in the Fibonacci sequence.  
\end{proposition}

\begin{proof}
Since $N(1,1)=1$ and $N(m,1)=0$ for all $m\ne 1$, it is true for $n=1$. Let $n_0$ be a positive integer less than $2p-1$. Suppose the inequality holds for all $n<n_0$. Observe that if $m=cp+1$ with $c\in\mathbb{N}$, then Lemma \ref{properties of N(m,n,p)}(b) and inductive assumption give us
\[ N(m,n_0)=N(cp,n_0-1)\le F(n_0-1). \]
Hence, we assume that $p\mid m$. In particular, suppose
\[ m=c_1p+c_2p^2+\cdots+c_{r_m}p^{r_m}\]
for $0\le c_i\le p-1$. Since $n_0\le 2p-2$, $N(m,n_0)$ is exactly the number of solutions of equations $a_i+b_{i+1}=c_i$ for $1\le i\le r_m$. So if $c_1=0$, we have from the formula \eqref{recursive formula} that 
\[ N(m,n_0)=N\left(\frac{m}{p},n_0\right). \]
Therefore, we can further assume that $c_1>0$. It follows that the equation $a_1+b_2=c_1$ has two solutions $a_1=c_1, b_2=0$ or $a_1=c_1-1, b_2=1$. So the recursive formula \eqref{recursive formula} and Lemma \ref{properties of N(m,n,p)}(b) imply that
\begin{align*}
N(m,n_0) &= N\left(\frac{m}{p}-c_1,n_0-2c_1\right)+N\left(\frac{m}{p}-c_1+1,n_0-2c_1+2\right)\\
 &= N\left(\frac{m}{p}-c_1,n_0-2c_1\right)+N\left(\frac{m}{p}-c_1,n_0-2c_1+1\right)\\
&\le F(n_0-2c_1)+F(n_0-2c_1+1).
\end{align*}
Since $c_1\ge 1$, the last sum is less than or equal to $F(n_0)$. This finishes the induction.
\end{proof}

\begin{remark}\label{Fibonacci conjecture}
It seems possible to remove the restriction $n\le 2p-2$. We have experimental evidence checking large values of $n$ and small values of $p$. For all $n$, estimating the value of $N(m,n)$ is very difficult, classified as a classical problem about a special type of partition function in number theory \cite[Chapter 7]{Gr:1984}.    
\end{remark}

It is shown from the above proposition that for a fixed $n\le 2p-2$, $N(m,n)$ is bounded as a function of $m$. So it is interesting to consider questions related to the maximum value of $N(m,n)$. In what follows we determine exactly the form of $m$ giving $\max\{N(m,n):m\in\mathbb{N}\}$.

\begin{proposition}\label{maximum form}
Let $n\le 2p-2$. If the maximum of $N(m,n)$ for all $m\in\mathbb{N}$ is attained then 
\begin{align}\label{max form of m}
m= \sum\limits_{i=1}^dp^{s_i}
\end{align}
for some $d>0$ and distinct positive integers $s_i$.
\end{proposition}

\begin{proof}
It suffices to show that for any non-negative integer $m'$, there is an $m$ of the form \eqref{max form of m} such that $N(m',n)\le N(m,n)$. Indeed, note that if $m'\equiv 1$ (mod $p$) then by Lemma \ref{properties of N(m,n,p)}(b) we have 
\[N(m',n)=N(m'-1,n-1)<\max_{m\in\mathbb{N}}\{N(m,n)\}.\]
So we assume $p\mid m'$. Suppose
\[ m'=\sum\limits_{i=1}^{d}c_ip^{d_i} \]
is the $p$-adic expansion of $m'$ with $1\le c_i\le p$. Then $\mathrm{ht}(m')=\sum_{i=1}^dc_i$ and $r_{m'}-s_{m'}=d$. Hence, Lemma \ref{key_lemma} gives us
\[ N(m',n)=\binom{d}{2\mathrm{ht}(m')-n}. \]
Now consider \[ m=\sum\limits_{j=1}^{\mathrm{ht}(m')}p^j. \]
Then $\mathrm{ht}(m)=\mathrm{ht}(m')$ and $r_m-s_m=\mathrm{ht}(m')\ge d$. Therefore, Lemma \ref{key_lemma} implies that
\[ N(m',n)\le N(m,n), \]
which proves our proposition.
\end{proof}

\subsection{Partition Functions and An Exponential Bound}\label{partitions}

We assume only for this part that $p$ is an arbitrary prime. The main task in this subsection is to give a bound of $N(m,n)$ when we fix $n$ and let $m$ vary. More precisely, we prove that the sequence $\left(\max\{N(m,n):m\in\mathbb{N}\}\right)$ is bounded exponentially from above. To this end, we shall first consider a special partition function and prove that it is bounded exponentially. 

Following the notation and definition in \cite[Chapter 7]{Gr:1984}, for any subset $A\subset \mathbb Z^+$, $m,n\in\mathbb Z^+$, we let
\[P_{A,n}(m):=\left\{(c_1,\ldots,c_r)\in(\mathbb{Z}^+)^r\colon m=\sum_{i=1}^r c_ia_i, a_i<a_{i+1}, a_i\in A, n=\sum_{i=1}^rc_i, r\in\mathbb Z^+\right\},\]
the set of partitions of $m$ into $n$ parts that belong to $A$, and let \[P_A(m)=\bigcup_{n\in\mathbb{Z}^+}P_{A,n}(m).\] 
We denote by $p_{A,n}(m)$ and $p_A(m)$ the cardinalities of $P_{A,n}(m)$ and $P_A(m)$. In other words, $p_{A,n}(m)$ is the number of ways in writing $m$ as a sum of $n$ elements of $A$. In particular, if $A=\mathbb Z^+$, then $p_{A}(m)$ (resp. $p_{A,n}(m)$) coincides with the number of traditional partitions of $m$ (resp. in $n$ parts). Hence, in this case, we write $p_n(m)=p_{A,n}(m)$ and $p(m)=p_A(m)$. It was Hardy and Ramanujan who first proved the celebrated asymptotic formula for $p(n)$:
\begin{equation*}
p(n)\sim \frac{1}{4n\sqrt{3}}e^{\pi\sqrt{2n/3}},\quad n\rightarrow\infty.\end{equation*} For our purpose, we will only need the bound \cite[Chapter 7, Theorem 10]{Gr:1984}:
\begin{equation}\label{partition bound}
p(n)<e^{\pi\sqrt{2n/3}},\quad\text{ for all } n\in\mathbb Z^+.\end{equation}

From now on, let $A=\{p^n\colon n=0,1,2,\ldots\}$. We shall only be interested in bounding $\sup_{m}p_{A,n}(m)$ exponentially and shall not try to obtain sharp bounds. Different types of partitions functions have been studied in the literature but mainly the asymptotic behavior with regard to $n$, see for example \cite[page 111]{Gr:1984}. 

An \emph{ordered} partition of $n$ is a list of pairwise disjoint nonempty subsets of a partition $\mathfrak a$ of $n$ such that the union of these subsets is $\mathfrak a$. We use the notation $\mathfrak c\Vdash n$ to denote that $\mathfrak c$ is an ordered partition of $n$. Let $p_o(n)$ denote the number of ordered partitions of $n$. For each $\mathfrak c=(c_1,c_2,\ldots,c_r)\Vdash n$, denote
\[B_{A,\mathfrak{c}}(m)=\left\{\mathfrak s=(s_1,s_2,\ldots,s_r)\colon m=\sum_{i=1}^rc_ip^{s_i}, 0\leq s_1<s_2<\ldots <s_r\right\},\]
and \[B_{A,n}(m)=\bigcup_{\mathfrak c\Vdash n}\left\{(\mathfrak c,\mathfrak s)\colon \mathfrak s\in B_{A,\mathfrak c}(m)\right\}.\]
It is easy to see that $p_{A,n}(m)$ is equal to the cardinality of $B_{A,n}(m)$. In order to exponentially bound $\sup_m p_{A,n}(m)$, we need two lemmas. 
\begin{lemma}\label{ordered partition}
For all $n\ge 1$, we have $p_o(n)\leq p(2n^2)<e^{2\pi n/\sqrt{3}}$.
\end{lemma}
\begin{proof}
Note that the second inequality follows from \eqref{partition bound}. For each $n\ge 1$, let 
\begin{align*}
B(n) &= \left\{(c_1,c_2,\ldots,c_r)\colon \sum_{i=1}^r c_i=n, 1\leq c_i\leq c_{i+1},r\in\mathbb Z^+ \right\},\\
B_o(n) &= \left\{(c_1,c_2,\ldots,c_r)\colon \sum_{i=1}^r c_i=n, 1\leq c_i, r\in\mathbb Z^+ \right\},
\end{align*}
and let $f$ be a map from $B_o(n)$ to $B(2n^2)$ defined by
\[(c_1,c_2,\ldots,c_r)\mapsto (c_1',c_2',\ldots, c_r')=\left(c_1,c_1+c_2,\ldots, \sum_{i=1}^{r-1}c_i, 2n^2-d\right),\]
where $d=\sum_{i=1}^{r-1}c_i'$. The reader can easily verify that $f$ is well-defined and injective. Therefore, we have shown that
\[ p_o(n)=|B_o(n)|\le|B(2n^2)|=p(2n^2),\]
which completes our proof.
\end{proof}

\begin{lemma}\label{independence-m} For all $n\ge 1$,
\[\sup_{m\in\mathbb{Z}^+,\mathfrak c\Vdash n}\left|B_{A,\mathfrak{c}}(m)\right|\leq 2^{n}.\]
\end{lemma}
\begin{proof} 
We proceed by induction on $n$. If $n=1$, then $\mathfrak c=(1)$. So for any $m$, $B_{A,\mathfrak{c}}(m)$ contains at most $1$ element, which proves the case $n=1$.

For arbitrary $n$, let $\mathfrak c\Vdash n$ and $m\in\mathbb{Z}^+$. We consider two cases:

{\bf Case 1:} Assume $p\nmid c_1$. If $(s_1,s_2,\ldots,s_r)$ and $(s_1',s_2',\ldots,s_r')$ are in $B_{A,\mathfrak{c}}(m)$, then
\[\sum_{i=1}^rc_ip^{s_i}=\sum_{i=1}^rc_ip^{s_i'},\]
from which it follows that $s_1=s_1'$. Therefore, \[|B_{A,\mathfrak{c}}(m)|\leq |B_{A,\mathfrak{c}'}(m-c_1p^{s_1})|\leq 4^{\frac{n-1}{p}}\leq 4^{n/p}\le 4^{n/2}=2^n\] by induction, where $\mathfrak c'=(c_2,c_3,\ldots,c_r)$.

{\bf Case 2:} Now we assume that $p\mid c_1$. Denote
\[\mathfrak c'=(c_i')=\left(\frac{c_1}{p},c_2,\ldots,c_r\right)\quad\text{and}\quad \mathfrak c''=(c_i'')=\left(\frac{c_1}{p}+c_2,c_3\ldots,c_r\right).\] It is clear that $n'=\sum c_i'=\sum c_i''<n$. We construct a map 
\[f\colon B_{A,\mathfrak{c}}(m)\rightarrow B_{A,\mathfrak{c}'}(m)\cup B_{A,\mathfrak{c}''}(m)\]
as follows: if $s_1+1=s_2$, $f(s_1,s_2,\ldots,s_r)=(s_2,s_3,\ldots,s_r)\in B_{A,\mathfrak{c}''}(m)$, and if $s_1+1<s_2$, $f(s_1,s_2,\ldots,s_r)=(s_1+1,s_2,s_3,\ldots,s_r)\in B_{A,\mathfrak{c}'}(m)$. It is not hard to see that $f$ is an injective well-defined map. Therefore, by induction,
\[
|B_{A,\mathfrak{c}}(m)| \leq |B_{A,\mathfrak{c}'}(m)|+|B_{A,\mathfrak{c}''}(m)|\leq 2\cdot 4^{n'/p}.
\]
But $n'=n-c_1(1-1/p)$, so
\[\frac{n'}{p}=\frac{n}{p}-\frac{c_1}{p}\left(1-\frac{1}{p}\right)\leq \frac{n}{p},\]
since $c_1\geq 1$ and $p\geq 2$. It follows that $|B_{A,\mathfrak{c}}(m)| \leq 4^{n/p}\le 2^n$. 
\end{proof}

\begin{theorem}\label{partition-sup}
For all $n\ge 1$, we have $\sup_m p_{A,n}(m)$ is bounded exponentially. More precisely,
\[\sup_m p_{A,n}(m)<e^{2\pi n/\sqrt{3}}2^{n}.\]
\end{theorem}
\begin{proof}
By Lemma \ref{ordered partition}, there are less than $e^{2\pi n/\sqrt{3}}$ ordered partitions $(\mathfrak c)$ of $n$. For such a $\mathfrak c$, $|B_{A,\mathfrak{c}}(m)|\leq 2^{n}$ for any $m$ by Lemma \ref{independence-m}. Hence, the theorem follows.
\end{proof}

\begin{remark}
The fact that $\sup_m p_{A,n}(m)$ is finite for given $n$ is essentially contained in the proof of Proposition 7.1.1 of \cite{EHP:2013}. This is related to the fact that $A$ is a special subset of $\mathbb Z^+$. In general, this is false. For example, in the case of $A=\mathbb Z^+$, for any $n\geq 2$, $\sup_mp_n(m)=\infty$. An interesting question is how to characterize sets $A$ for which $\sup_m p_{A,n}(m)$ is finite for each $n$. Furthermore, one may then investigate the asymptotic behavior of $\sup_m p_{A,n}(m)$, as $n\rightarrow\infty$.
\end{remark}

Before giving an upper bound for $\sup\{N(m,n):m\in\mathbb{N}\}$, we consider an extension of the system \eqref{system}. Given a sequence of positive integers $\mathfrak d=(d_i)$. For all $m,n\ge 0$ and $r\ge 1$, let $N_r^{\mathfrak d}(m,n)$ the the number of solutions $(\mathfrak a, \mathfrak b')$ with $\mathfrak a=(a_1,\ldots,a_r)\in\mathbb{N}^r$ and $\mathfrak b'=(b'_1,\ldots,b'_r)$ and $b'_i\in\{0, d_i\}$ satisfying the following system
\begin{align}\label{extended system}
\begin{cases}
\displaystyle{2\sum_{i=1}^r a_i+\sum_{j=1}^rb'_j=n},\\
\displaystyle{b'_1+\sum_{i=1}^{r-1}(a_i+b'_{i+1})p^i+a_rp^r=m.}
\end{cases} 
\end{align}
Note that when $\mathfrak d=(1,1,\ldots,1)$ we have $N_r^{\mathfrak d}(m,n)=N_r(m,n)$. This extension will be used in Section \ref{applications} to compute an upper bound for the dimension of the cohomology of finite group of Lie type $SL_2(p^s)$. 

\begin{proposition}\label{exponential bound for V(m)}
Let $\mathfrak d=(d_i)$ be a sequence of positive integers. For all $n\ge 1$,
\[\sup{\{N_r^{\mathfrak d}(m,n):m\in\mathbb{N}, r\in\mathbb{Z}^+\}}\le n4^{n}e^{2\pi n/\sqrt{3}}.\]
In particular, $\sup{\{N(m,n):m\in\mathbb{N}\}}$ is bounded exponentially by $n4^{n}e^{2\pi n/\sqrt{3}}$.
\end{proposition}
\begin{proof}
Let $S'(m,n)$ be the set of solutions of systems \eqref{extended system} for all $r\ge 1$. We construct a map
\[f\colon S'(m,n)\rightarrow \bigcup_{1\leq j\leq n}B_{A,j}(m)\] as follows: Let $(\mathfrak a,\mathfrak b')\in S'(m,n)$ with $\mathfrak a=(a_i)$ and $\mathfrak b'=(b'_i)$. Consider the vector $(a_0+b'_1,a_1+b'_2,a_2+b'_3,\ldots)$ where $a_0=0$. We define $\mathfrak s$ to be the vector of indices $j$, in the increasing order, such that $a_j+b'_{j+1}>0$. We define $\mathfrak c=(c_i)$ by $c_i=a_{s_i}+b'_{s_i+1}$ and $(\mathfrak c,\mathfrak s)$ lands in $B_{A,j}(m)$ with $j=\sum_i c_i\leq n$. Let $f(\mathfrak a,\mathfrak b')=(\mathfrak c,\mathfrak s)$.

For each pair $(\mathfrak c,\mathfrak s)$, from the construction of $f$, we see that there are at most $2^r\leq 2^n$ elements in the fiber $f^{-1}(\mathfrak c,\mathfrak s)$, where $r$ is the length of $\mathfrak c$. By Theorem \ref{partition-sup}, it follows that 
\[\sup{\{N_r^{\mathfrak d}(m,n):m\in\mathbb{N},r\in\mathbb{Z}^+\}}=\sup_m|S'(m,n)|\leq 2^n\sup_m\left|\bigcup_{1\leq j\leq n}B_{A,j}(m)\right|\leq n4^ne^{2\pi n/\sqrt{3}}.\]
Consequently, we have
\[ \sup{\{N(m,n):m\in\mathbb{N}\}}\le n4^{n}e^{2\pi n/\sqrt{3}},\]
which finishes our proof.
\end{proof}

\begin{remark}\label{modification of (b_i)} For brevity, we denote from now on \[ C_n= n4^{n}e^{2\pi n/\sqrt{3}}.\]
\end{remark}
 
\section{Cohomology of Weyl modules}\label{cohomology}

We keep assuming $p$ is an odd prime in this section, unless otherwise stated. Recall that we identify the weight lattice $X(T)$ with $\mathbb{Z}$, and the root lattice $\mathbb{N}\Phi$ with $2\mathbb{Z}$. In this section, we study the dimension of $\opH^n(G, V(m))$, where $m\in X^+=\mathbb{N}$. It is well known that $\opH^n(G,V(m))=0$ for $n\ge 0$ if $m$ is not in the root lattice, see for example \cite[page 100]{EHP:2013}.

We start with computing the dimension of the cohomology space $\opH^n(B,-m)$ for each $m,n\ge 0$, which is isomorphic to $\opH^n(B_rT,-m)$ for sufficiently large $r$, see \cite[Section 7]{CPS}. It is worth noting that calculating $B$-cohomology, in general, is an important task due to its connection with cohomology of symmetric groups, see details in \cite{HN:2006}. The following lemma relates the latter with $N(m,n)$ studied in the previous section.

\begin{lemma}\label{B_rT-cohomology}
For all $m,n\ge 0$, we have \[ \dim\opH^n(B_rT,-m)=N_r\left(\frac{m}{2},n\right),\] for all $r\ge 1$. Therefore, taking sufficiently large $r$, we have
\[ \dim\opH^n(B,-m)=N\left(\frac{m}{2},n\right).\]
\end{lemma} 

\begin{proof}
Applying Propositions 4.1.1 and 4.2.1 in \cite{Ngo:2013} for $\lambda=-m$, we have each weight of $\opH^n(B_r,-m)$ is of the following form
\[
p^r\left[a_r+\frac{\frac{\frac{\frac{\frac{-m}{2}+b_1}{p}+a_1+b_2}{p}+a_2+b_3}{p}+\cdots+a_{r-1}+b_r}{p}\right]\alpha
\]
or 
\begin{align}\label{continued fraction}
\left[p^ra_r+p^{r-1}(a_{r-1}+b_r)+\cdots+p(a_1+b_2)+b_1-\frac{m}{2}\right]\alpha
\end{align}
with $2(a_1+\cdots+a_r)+b_1+\cdots+b_r=n$. Since $B_rT/B_r\cong T^{(r)}$, the Lyndon-Hochschild-Serre spectral sequence collapses and gives
$$\opH^n(B_rT,-m)\cong\opH^n(B_r,-m)^{T}=\opH^n(B_r,-m)_0,$$ the zero weight space of $\opH^n(B_r,-m)$. Hence, the dimension of $\opH^n(B_rT,-m)$ is the number of tuples $(a_1,\ldots,a_r,b_1,\ldots,b_r)$ such that $\displaystyle{n=\sum_{i=1}^r(2a_i+b_i)}$ and the form \eqref{continued fraction} is zero. This number is exactly $N_r\left(\frac{m}{2},n\right)$ as desired. Therefore, the equality \[\dim\opH^n(B,-m)=N\left(\frac{m}{2},n\right)\] follows from \cite[II.4.12]{Jan:2003} when taking sufficiently large $r$.
\end{proof}

We now turn our attention to the $G$-cohomology of Weyl modules $V(m)$.

\begin{theorem}\label{cohomology in V(m)}
For all $m,n\ge 0$, we have
\[ \opH^n(G,V(m))\cong\opH^{n+1}(B,-m-2). \]
Consequently, $\dim\opH^n(G,V(m))=N\left(\frac{m}{2}+1,n+1\right)$. In particular,
\[  \dim\opH^n(G,V(m))\ge \binom{r_{\frac{m}{2}+1}-s_{\frac{m}{2}+1}}{2\mathrm{ht}(\frac{m}{2}+1)-c_0-n-1}.\]
For $n\le 2p-3$, equality holds and $\dim\opH^n(G,V(m))$ attains the maximum if 
\begin{align*}
m=2\sum\limits_{i=1}^dp^{s_i}-2
\end{align*}
for some $d>0$ and distinct positive integers $s_i$.
\end{theorem}

\begin{proof}
From Proposition I.4.5(b) in \cite{Jan:2003}, we have the spectral sequence
\[ \opH^{i}(G,R^j\ind_B^G(-m-2))\Rightarrow \opH^{i+j}(B,-m-2). \]
Since $R^j\ind_B^G(-m-2)=0$ for all $j\ne 1$, by \cite[Proposition II.5.2(d)]{Jan:2003}, the spectral sequence collapses and gives us for each $n\ge 0$
\[ \opH^{n}(G,R^1\ind_B^G(-m-2))\cong \opH^{n+1}(B,-m-2). \]
Now Serre duality (see for example \cite[II.5.2]{Jan:2003}) gives us that $$R^1\ind_B^G(-m-2)\cong\opH^0(m)^*=V(m).$$
So we have proved \[ \opH^n(G,V(m))\cong\opH^{n+1}(B,-m-2). \] 
The other statements of the theorem follow from Lemmas \ref{key_lemma}, \ref{B_rT-cohomology} and Proposition \ref{maximum form}.
\end{proof}

Recall that Erdmann, Hannabuss, and Parker in their recent paper \cite{EHP:2013} prove that the sequence $\left(\max{\{\dim\opH^n(G,V(m)):m\in\mathbb{N} \}}\right)$ has an exponential lower bound and an upper bound including $n!$. This lower bound then implies that the sequence grows at least exponentially. Their method relies heavily on generating functions and labels, and requires lengthy calculations. We provide a much shorter argument using Theorem \ref{cohomology in V(m)} and properties of $N(m,n)$. Moreover, we establish a better upper bound.

\begin{corollary}\label{recover EHP}
For all $m\ge 0$ and $n\ge 1$, \[\dim\opH^n(G,V(m))\le C_{n+1}.\] Consequently, the sequence $\left(\max{\{\dim\opH^n(G,V(m)):m\in\mathbb{N} \}}\right)$ grows exponentially.
\end{corollary}

\begin{proof} 
The first statement immediately follows from Theorem \ref{cohomology in V(m)} and Proposition \ref{exponential bound for V(m)}. For the remainder, we prove the sequence $\left(\max\{N(m,n+1): m\in\mathbb{Z}^+\}\right)$ grows exponentially. We first compute an exponential lower bound. Suppose $3\mid(n+2)$, set $t=\frac{n+2}{3}-1$, and consider
\[ m=2p+2p^2+\cdots+2p^{2t+1}-2. \]
It then follows from Theorem \ref{cohomology in V(m)} that  
\[ \dim\opH^n(G,V(m))\ge\binom{2t}{t}\ge 2^t=2^{\frac{n+2}{3}-1}. \]
Finally the argument will be completed if we show that the sequence \[ \left(\max{\{N(m,n+1):m\in\mathbb{N} \}}\right)\]
is not decreasing. Indeed, for each $n\ge 1$, choose $m_0$ and $r$ such that  \[N(m_0,n)=N_r(m_0,n)=\max\{N(m,n):m\in\mathbb{N}\};\]
for example, we may take $r>\log_pm_0$. The map $S_r(m_0,n)\rightarrow S_{r+1}(m_0+p^{r+1},n+1)$, defined by
\[((a_1,\ldots,a_r),(b_1,\ldots,b_r))\mapsto ((a_1,\ldots,a_r,0),(b_1,\ldots,b_r,1)),\]
is clearly well-defined and injective.
It follows that 
\[N_r(m_0,n)\le N_{r+1}(m_0+p^{r+1},n+1)\le N(m_0+p^{r+1},n+1)\le\max\{N(m,n+1):m\in\mathbb{N}\},\]
hence
\[\max\{N(m,n):m\in\mathbb{N}\}\le\max\{N(m,n+1):m\in\mathbb{N}\}.\]
This finishes the proof.
\end{proof}

\begin{remark}\label{p=2}
Our method also works for the case when $p=2$. In such case, using results in \cite[Section 4.4]{Ngo:2013} and similar arguments as in Lemma \ref{B_rT-cohomology} and Theorem \ref{cohomology in V(m)}, we can compute the dimension of $\opH^n(G,V(m))$. Notably, as the $U_r$-cohomology does not depend on $\Lambda^\bullet(\frak u^*)$, the exterior algebra over $\fraku^*$, when $p=2$, the dimension of $\opH^n(G,V(m))$ is then given by the number of solutions $(a_1,\ldots,a_r)\in\mathbb{N}^r$ of a much simpler system
\begin{align*}
\begin{cases}
a_1+\cdots+a_r=n+1,\\
2a_1+\cdots+2^ra_r=m+1,
\end{cases}
\end{align*}
for sufficiently large $r$, which we denote by $N'(m+1,n+1)$. Observe that $$N'(m+1,n+1)=p_{A,n+1}(m+1)$$ with $A=\{2, 2^2,2^3,\ldots\}$, described in Subsection \ref{partitions}. So $N'(m+1,n+1)$ coincides with the formula in \cite[Corollary 3.2.2]{EHP:2013}.

In principle, the ideas in Section \ref{counting} can be applied to investigate properties of $N'(m,n)$. More precisely, to employ the $p$-adic expansion, we need $n$ to be less than $p$, so that $n$ is either $0$ and $1$. Then our method produces the same result for the space $\Ext^1$ as computed in \cite{E:1995}. However, for arbitrary $n$, our study on partition functions in Subsection \ref{partitions} gives a nice upper bound for $N'(m,n)$, which is much smaller than the bound $C_n$ for $N(m,n)$. In particular, applying Theorem \ref{partition-sup} we have  
\end{remark}
\begin{theorem}\label{bound for p=2}
Suppose $p=2$. For all $n\ge 1$ and $m\ge 0$,
\[ \dim\opH^n(G,V(m))\le e^{2\pi(n+1)/\sqrt{3}}2^{n+1}=\frac{C_{n+1}}{(n+1)2^{n+1}}. \]
\end{theorem}

A further advantage of our method is the feasibility of explicitly computing low degree cohomology, which is a challenging task in cohomology theory even in the $SL_2$-case. Note that, from the binomial formula in Theorem \ref{cohomology in V(m)}, we could compute the dimension of $\opH^n(G,V(m))$ for all $n\le 2p-3$. However, as calculations will involve a large number of cases for $n>3$, we only list below the detailed descriptions for $\dim\opH^n(G,V(m))$ with $n\le 3$. From now on, we assume again that $p$ is an odd prime.

\begin{proposition}\label{low degree cohomology of V(m)}
For all $m\ge 0$, we have
\begin{enumalph}
\item 
 $~~\dim\opH^1(G,V(m))= \dim\opH^2(B,-m-2)\\\text{~~~~~~~~~~~~~~~}
~\quad\quad\quad\quad\quad\quad\quad= 
\begin{cases}
1~~~ &\text{if}~~m\in\{2p^u-2\epsilon, 2p^u+2p^v-2\},\\
0~~~&\text{otherwise},
\end{cases}
$
\item
$~~\dim\opH^2(G,V(m))=\dim\opH^3(B,-m-2)\\\text{~~~~~~~~~~~~~~~}
~\quad\quad\quad\quad\quad\quad\quad
=\begin{cases}
2~~~ &\text{if}~~m=2p^u+2p^v-2,\\
1~~~ &\text{if}~~m\in\{2p^u, 4p^u-2, 2p^u+2p^v, 2p^u+2p^v+2p^w-2\},\\
0~~~&\text{otherwise},
\end{cases}
$
\item
$~~\dim\opH^3(G,V(m))=\dim\opH^4(B,-m-2)\\\text{~~~~~~~~~~~~~~~}
~\quad\quad\quad\quad\quad\quad\quad
=\begin{cases}
3~~~ &\text{if}~~m=2p^u+2p^v+2p^w-2,\\
2~~~ &\text{if}~~m=2p^u+2p^v,\\
1~~~ &\text{if}~~m\in\{4p^u-2\epsilon, 2p^u+2p^v-2, 4p^u+2p^v-2, \\
& 2p^u+2p^v+2p^w, 2p^u+2p^v+2p^w+2p^x-2\},\\
0~~~&\text{otherwise},
\end{cases}
$
\end{enumalph}
where the exponents u, v, w, and x are distinct positive integers and $\epsilon$ is either 0 or 1.
\end{proposition}

\begin{proof}
Theorem \ref{cohomology in V(m)} implies that for $n=1,2,3$,
\[ \dim\opH^n(G,V(m))= \binom{r_{\frac{m}{2}+1}-s_{\frac{m}{2}+1}}{2\mathrm{ht}(\frac{m}{2}+1)-c_0-n-1}.\]
Observe that in order for $\dim\opH^n(G,V(m))$ to be nonzero, we need
\[ c_0+n+1\le 2\mathrm{ht}\left(\frac{m}{2}+1\right)\le r_{\frac{m}{2}+1}-s_{\frac{m}{2}+1}+c_0+n+1. \]
Note also that $\mathrm{ht}(\frac{m}{2}+1)\ge r_{\frac{m}{2}+1}-s_{\frac{m}{2}+1}+c_0$. So the above inequalities can be deduced to
\[ \frac{c_0+n+1}{2}\le \mathrm{ht}\left(\frac{m}{2}+1\right)\le n+1. \]
It is now easy to work out all the values of $\mathrm{ht}(\frac{m}{2}+1)$ for each $n=1,2,3$ so that we obtain the desired formulas of $m$ as in the theorem.
\end{proof}

\begin{remark}
Our results for the second degree and third degree $B$-cohomology in the preceding proposition agree with the calculations of Bendel, Nakano, and Pillen in \cite{BNP1}\cite{BNP2} when specialized to $G=SL_2$. The computation for $\opH^4(B,-m-2)$ is new. Next we give an upper bound for the dimension of $\opH^n(G,V(m))$ for degree $n$ up to $2p-3$.
\end{remark}

\begin{proposition}\label{Fibonacci for V(m)}
For all $n\le 2p-3$ and $m\ge 0$, we have $\dim\opH^n(G,V(m))\le F(n+1)$, where $F(n+1)$ is the $(n+1)$-th Fibonacci number.
\end{proposition}

\begin{proof}
This follows from Proposition \ref{Fibonacci} and Theorem \ref{cohomology in V(m)}.
\end{proof}

\section{Extensions of Weyl modules}\label{extension}
In this section we combine our calculations in previous sections and the recursive formulas in \cite{Pa:2006} to compute the dimension of $\Ext^3_G(V(m_2),V(m_1))$. We also obtain upper bounds for the dimension of higher extensions of Weyl modules. Again, $p$ is assumed to be an odd prime.

Throughout the section, we let $m_1, m_2$ be non-negative integers where $m_1=pa+i$ and $m_2=pb+j$ with $a, b, i, j\in\mathbb{N}$ and $0\le i, j\le p-1$. Without loss of generality, we assume that $a\ge b$. Using blocks of $SL_2$, we can determine the values of $m_1, m_2$ such that the extension space between $V(m_1)$ and $V(m_2)$ might be nonzero. More explicitly, if $i=p-1$ (resp. $j=p-1$), then $m_1$ and $m_2$ are in the same block only if $j=p-1$ (resp. $i=p-1$). Moreover, we have for each $n\ge 0$
\[ \Ext^n_G(V(pb+p-1),V(pa+p-1))\cong\Ext^n_G(V(b),V(a)).\]
For this reason, we assume from now on that $0\le i, j\le p-2$. Then there are only two cases when $\Ext^n_G(V(m_2),V(m_1))$ could be nonzero:
\begin{itemize}
\item $a-b$ is even and $i=j$,
\item $a-b$ is odd and $j=p-2-i$.
\end{itemize}
We refer the reader to \cite[page 388]{Pa:2006} for further details.

\subsection{Low degree extensions} Recall that $\Ext^1$ and $\Ext^2$ spaces were described by Erdmann et. al. in \cite{E:1995} and \cite{CE}. It follows from their results that
\[ \dim\Ext_G^n(V(m_2),V(m_1))\le n \] 
for $n=1,2$. We now extend this result to $n=3$ and then give an upper bound for the dimension of $\Ext^n$ for $n\le 2p-3$. 

We begin with an explicit description for $\Ext^3$. 

\begin{theorem}\label{Ext^3}
If $a-b$ is even, then
\begin{align}\label{Ext^3(1)}
 \dim\Ext_G^3(V(pb+i),V(pa+i))=
\begin{cases}
1~~&\text{if}~~a-b=4,\\
N\left(\frac{a-b}{2},3\right) &\text{otherwise}.\\
\end{cases}
\end{align}
Otherwise, if $a-b$ is odd, then set $j=p-2-i$ and we have
\begin{align*}
\dim\Ext_G^3(V(pb+j),V(pa+i))=
\begin{cases}
1~~&\text{if}~~a-b=2p^u+3-2\epsilon,\\
2~~&\text{if}~~a-b=2p^u+2p^v+1,\\
&~\text{~~~~with}~u=1~\text{or}~v=1, \\
3~~&\text{if}~~a-b=2p^u+2p^v+1, \\
&~~~\text{~~~with}~u, v>1,\\
\dim\Ext_G^3(V(b),V(a-1)) &\text{otherwise},\\
\end{cases}
\end{align*}
where $u, v$ are distinct positive integers and $\epsilon$ is either $0$ or $1$. 
\end{theorem} 

\begin{proof}
Suppose $a-b$ is even. From Theorem 5.1 in \cite{Pa:2006}, we have
\[
\Ext_G^3(V(pb+i),V(pa+i))\cong \opH^2(G,V(a-b-2))\oplus\opH^0(G,V(a-b-4)).
\]
Note that $\dim\opH^0(G,V(a-b-4))=1$ if and only if $a-b=4$. Then $\opH^2(G,V(a-b-2))=0$, which proves the first equality. When $a-b\ne 4$, that is $\dim\opH^0(G,V(a-b-4))=0$, we obtain
\[ \dim\Ext_G^3(V(pb+i),V(pa+i))=\dim\opH^2(G,V(a-b-2))=N\left(\frac{a-b}{2},3\right). \]
 
Now suppose $a-b$ is odd. Again, using \cite[Theorem 5.1]{Pa:2006} we have
\begin{align}\label{Ext^3(3)}
\Ext_G^3(V(pb+j),V(pa+i))\cong \opH^1(G,V(a-b-3))\oplus\Ext_G^3(V(b),V(a-1)).
\end{align}
Applying Proposition \ref{low degree cohomology of V(m)}(a) to $\opH^1(G,V(a-b-3))$, we get the following three cases:

{\bf Case 1:} $a-b-3=2p^u-2\epsilon$ with $u$ a positive integer and $\epsilon$ either $0$ or $1$. Then $\dim\opH^1(G,V(a-b-3))=1$. Suppose $b=pb'+i'$ with some non-negative integers $b'$ and $i'\le p-2$, then \[ a-1=p(2p^{u-1}+b')+i'+2-2\epsilon.\]
Now if $\epsilon=0$, then $a-1$ and $b$ are not in the same block; hence $\Ext_G^3(V(b),V(a-1))\cong 0$. Otherwise, $\epsilon=1$, using \eqref{Ext^3(1)} we have 
\[ \dim\Ext_G^3(V(b),V(a-1))=N\left(p^{u-1},3\right)=0.\]
Therefore, \eqref{Ext^3(3)} implies that
\[ \dim\Ext_G^3(V(pb+j),V(pa+i))=\dim\opH^1(G,V(a-b-3))=1.\]  

{\bf Case 2:} $a-b-3=2p^u+2p^v-2$ with $u, v$ are distinct positive integers. Then $\dim\opH^1(G,V(a-b-3))=1$. It follows that $a-1=2p^u+2p^v+b$. A similar argument as in the previous case together with \eqref{Ext^3(1)} give us
\begin{align*}
\dim\Ext_G^3(V(b),V(a-1))&= N(p^{u-1}+p^{v-1},3)\\
&= 
\begin{cases}
1~~&\text{if}~~u=1~\text{or}~v=1,\\
2~~&\text{else}.
\end{cases}
\end{align*}
So we have from \eqref{Ext^3(3)} that
\begin{align*}
\dim\Ext_G^3(V(pb+j),V(pa+i))=
\begin{cases}
2~~&\text{if}~~a-b=2p^u+2p^v+1,\\
&~\text{~~~~with}~u=1~\text{or}~v=1, \\
3~~&\text{if}~~a-b=2p^u+2p^v+1, \\
&~~~\text{~~~with}~u, v>1.
\end{cases}
\end{align*}

{\bf Case 3:} If $a$ and $b$ are not in Case 1 or 2, then we know from Proposition \ref{low degree cohomology of V(m)}(a) that $\dim\opH^1(G,V(a-b-3))=0$; hence
\[
\Ext_G^3(V(pb+i),V(pa+i))\cong\Ext_G^3(V(b),V(a-1)),
\]
which finishes our proof.
\end{proof}

It follows immediately that

\begin{corollary}\label{Ext^3 bound}
For all $m_1, m_2\ge 0$, we have $\dim\Ext_G^3(V(m_2),V(m_1))\le 3$.
\end{corollary}

Next, we extend the result in Proposition \ref{Fibonacci for V(m)} to extension spaces between Weyl modules. 

\begin{lemma}\label{Fibonacci-bound-ext-lemma}
Suppose $a\ge b$ are non-negative integers and $0\le i,j\le p-2$. Then for $1\le n\le 2p-3$,
\begin{align}\label{Fibonacci-bound-ext-lemma1}
\dim\Ext^n_G(V(pb+i),V(pa+i))\le F(n+1),
\end{align}
and
\begin{align}\label{Fibonacci-bound-ext-lemma2}
\dim\Ext^n_G(V(j),V(pa+i))\le F(n+1).
\end{align}
\end{lemma}

\begin{proof}
We first prove the inequality \eqref{Fibonacci-bound-ext-lemma1}. As observed at the beginning of the section, we can assume $a-b$ is even. Then Theorem 5.1 in \cite{Pa:2006} gives us
\[  \dim\Ext^n_G(V(pb+i),V(pa+i))=\sum_{t=0}^{(a-b-2)/2}\dim\opH^{n-1-2t}(G,V(a-b-2-2t))\]
for $n\le\min\{2p-3,a-b-1\}$. Note that $\dim\Ext^n_G(V(pb+i),V(pa+i))\le 1$ for $n\ge a-b-1$. Now by Proposition \ref{Fibonacci for V(m)}, we have for $n\le 2p-3$
\[  \dim\Ext^n_G(V(pb+i),V(pa+i))=\sum_{t=0}^{\lfloor\frac{n}{2}\rfloor}F(n-2t)\le F(n+1),\]
which proves the inequality \eqref{Fibonacci-bound-ext-lemma1}.

To prove \eqref{Fibonacci-bound-ext-lemma2}, we only need to consider the case when $a$ is odd, $j=p-2-i$, and $0\le n\le a-1$. Then we have again from \cite[Theorem 5.1]{Pa:2006} and Lemma \ref{properties of N(m,n,p)}(a)(b),
\begin{align*}
 \dim\Ext^n_G(V(j),V(pa+i))&=\sum_{t=0}^{(a-1)/2}\dim\opH^{n-2t}(G,V(a-1-2t))\\
&=\sum_{t=0}^{\lfloor\frac{n}{2}\rfloor}N\left(\frac{a+1}{2}-t, n-2t+1\right)\\
&=\sum_{t=0}^{p-2}N\left(\frac{a+1}{2}-t, n-2t+1\right)\\
&=N\left(\frac{a+1}{2}-t_0, n-2t_0+1\right)+N\left(\frac{a+1}{2}-t_1, n-2t_1+1\right)\\
&=N\left(\frac{a+1}{2}-t_0, n-2t_0+1\right)+N\left(\frac{a+1}{2}-t_1-1, n-2t_1\right)
\end{align*}
where $0\le t_0, t_1\le p-2$, $\frac{a+1}{2}-t_0\equiv 0$ (mod $p$) and $\frac{a+1}{2}-t_1\equiv 1$ (mod $p$). If $t_0=0$ then there is no solution for $t_1$, so we obtain by Proposition \ref{Fibonacci} that
\[ \dim\Ext^n_G(V(j),V(pa+i))=N\left(\frac{a+1}{2}, n+1\right)\le F(n+1). \]
Otherwise, $t_0\ne 0$, and so
\[ \dim\Ext^n_G(V(j),V(pa+i))\le F(n-1)+F(n)=F(n+1).\]
\end{proof}
We are now ready to give an upper bound for the dimension of $\Ext^n$. 
\begin{theorem}\label{Fibonacci for Ext}
Suppose $m_1, m_2$ are non-negative integers and $m_2\in X^+_r$, i.e., $m_2< p^r$, for some $r\ge 1$. Then for all $n\le 2p-3$ 
\[ \dim\Ext^n_G(V(m_2), V(m_1))\le F(n+1)+(r-1)F(n). \]
\end{theorem}

\begin{proof}
By \eqref{Fibonacci-bound-ext-lemma2} in Lemma \ref{Fibonacci-bound-ext-lemma1}, the inequality of the theorem holds for $r=1$. Suppose now that it holds for some $r>1$. Let $m_2\in X^+_{r+1}$. Following from the notation at the beginning of this section, let
\begin{align*}
m_1=pa+i,\quad m_2=pb+j,
\end{align*}
where $a,b,i,j\in\mathbb{N}$ and $a\ge b$, $0\le i,j\le p-2$. From \eqref{Fibonacci-bound-ext-lemma1} in Lemma \ref{Fibonacci-bound-ext-lemma}, if $a-b$ is even and $i=j$, we obtain
\[ \dim\Ext^n_G(V(m_2), V(m_1))\le F(n+1). \]
We now consider the case when $j=p-2-i$ and $a-b$ is odd. It then follows from \cite[Corollary 5.2]{Pa:2006}, Lemma \ref{Fibonacci-bound-ext-lemma}, and inductive hypothesis that
\begin{align*}
\dim\Ext_G^n(V(m_2),V(m_1)) &=\dim\Ext^n_G(V(b),V(a-1))+\dim\Ext^{n-1}_G(V(pb+j),V(p(a-1)+j))\\
&\le F(n+1)+(r-1)F(n)+F(n)\\
&=F(n+1)+rF(n),
\end{align*}
which finishes our inductive argument.
\end{proof}

\subsection{Higher degree extensions} For $n> 2p-3$, we show that for each $r\ge 1$ the sequence $$\left(\max_{m_1\in X^+,m_2\in X_r^+}\{\dim\Ext^n_G(V(m_2),V(m_1))\}\right)$$ grows exponentially. It suffices to prove that the dimension of $\Ext^n_G(V(m_2),V(m_1))$ is bounded by an exponential function of $n$. As the proof is similar to that of Theorem \ref{Fibonacci for Ext}, we just outline the necessary modifications. Again, we write $m_1=pa+i$ and $m_2=pb+j$ with $a\ge b$ and consider two cases:

{\bf Case 1:} If $a-b$ is even and $i=j$ then by \cite[Theorem 5.1]{Pa:2006} and Proposition \ref{exponential bound for V(m)} we obtain an exponential bound for $\displaystyle{\max_{m_1\ge 0}\{\dim\Ext^n_G(V(m_2),V(m_1))\}}$. In particular, we have 
\begin{align*}
\dim\Ext^n_G(V(m_2),V(m_1)) &\le \sum_{j=0}^{\lfloor{\frac{n-1}{2}}\rfloor}\max_{m\ge 0}\{\dim\opH^{n-2t-1}(G,V(m))\}\\
&\le \sum_{j=0}^{\lfloor{\frac{n-1}{2}}\rfloor}C_{n-2t}\\
&\le C_{n+1}.
\end{align*}

{\bf Case 2:} If $a-b$ is odd and $j=p-2-i$, then \cite[Theorem 5.1]{Pa:2006} implies that
\begin{align*}
 \dim\Ext^n_G(V(j),V(pa+i))&=\sum_{t=0}^{(a-1)/2}\dim\opH^{n-2t}(G,V(a-1-2t))\\
&\le\sum_{t=0}^{\lfloor\frac{n}{2}\rfloor}C_{n+1-2t}\\
&\le C_{n+2}.
\end{align*}
Similar argument as for Theorem \ref{Fibonacci for Ext} gives us the following 

\begin{corollary}
For all $r\ge 1$, the sequence $\displaystyle{ \left(\max_{m_1\in X^+,m_2\in X_r^+}\{\dim\Ext^n_G(V(m_2),V(m_1))\}\right)}$ has exponential growth. More precisely, for all $m_1\in X^+, m_2\in X_r^+$
\[\dim\Ext^n_G(V(m_2),V(m_1))\le C_{n+2}+(r-1)C_n. \]
\end{corollary}  

We end this section with the following 

\begin{question}
Suppose $G=SL_2$. For all $n\ge 0$ and $r\ge 1$, is it true that for $m_1\in X^+$ and $m_2\in X_r^+$,$$\dim\Ext^n_G(V(m_2),V(m_1))\le F(n+1)+(r-1)F_n?$$
\end{question} 

\section{Applications}\label{applications}

We present here a few interesting applications, for example, in the cohomological theory of symmetric groups, finite group of Lie type $SL_2(p^s)$, and the algebraic group $SL_2$.

\subsection{Higher rank group $G$ and cohomology of Specht modules} Suppose that $G$ is a reductive group defined over $k$. Let $\Pi$ be the set of simple roots. For each root $\alpha$ in $\Phi$, we denote by $\alpha^\vee$ the dual root corresponding to $\alpha$. Let $(\cdot,\cdot)$ be the inner product on the Euclidean space $E:= \mathbb{Z}\Phi\otimes_{\mathbb{Z}}\mathbb{R}$. Suppose $\lambda$ and $\mu$ are dominant weights in $X^+$. If the difference $\mu-\lambda$ is a multiple of some simple root in $\Pi$, the extension spaces between $V(\lambda)$ and $V(\mu)$ can be related to those for $SL_2$ via the following result.

\begin{proposition}\cite[Corollary 4.3]{E:1995}\label{general G}
Let $\lambda, \mu\in X^+$ and $\mu-\lambda=d\alpha$ for some integer $d$ and simple root $\alpha$. Suppose that $m_{\lambda}=(\lambda,\alpha^\vee)$ and $m_{\mu}=(\mu,\alpha^\vee)$. Then for all $n\ge 0$
\[ \Ext_G^n(V(\lambda),V(\mu))\cong\Ext_{SL_2}^n(V(2m_\lambda),V(2m_\mu)). \]
\end{proposition}

Combining this result with Corollary \ref{recover EHP} for the extension spaces between Weyl modules, we can make the following statement about the growth rate of the dimension of these spaces.

\begin{proposition}
The growth rate of the sequence $\displaystyle{\left(\max_{\lambda, \mu\in X^+} \dim\Ext_G^n(V(\lambda), V(\mu))\right)}$ is at least exponential.
\end{proposition} 
 
In the case when $G$ is a general linear group, Proposition \ref{general G} has a direct connection to the low degree cohomology of Specht modules of symmetric groups. This relationship was extensively studied in \cite{KN:2001} and \cite{He:2009}. We mainly use the notation from \cite{He:2009}. 

Let $\Sigma_d$ be the symmetric group on $d$ letters. For each partition $\lambda$ of $d$, denote $S^{\lambda}$ the Specht module of $\Sigma_d$. Kleshchev and Nakano prove in \cite{KN:2001} that for $0\le n\le 2p-4$
\[ \opH^n(\Sigma_d,S^{\lambda})\cong \Ext_{GL_d(k)}^n(\opH^0(d),\opH^0(\lambda))\cong  \Ext_{GL_d(k)}^n\left(V(-w_0\lambda),V(-w_0d)\right), \]
where $w_0$ is the longest element in the Weyl group of $GL_d(k)$, which is $\Sigma_d$, and $\opH^0(d):=\opH^0((d,0,\ldots,0))$. When $\lambda$ is a two-part partition, $\Ext_{GL_d(k)}^n\left(V(-w_0\lambda),V(-w_0d)\right)$ is isomorphic to the $n$-th extension space of the corresponding Weyl modules of $SL_2$, see \cite[5.2]{He:2009}. Hence, our results on the latter yield information about $\dim\opH^n(\Sigma_d,S^{\lambda})$ with $n\le 2p-4$. From Corollary \ref{Ext^3 bound}, Theorem \ref{Fibonacci for Ext}, and Proposition \ref{general G}, we obtain upper bounds of the dimension of these cohomology spaces.

\begin{theorem}\label{Specht module}
Suppose $\lambda$ is a two-part partition of $d$. Then for $n=1,2,$ or $3$, we have
\[ \dim\opH^n(\Sigma_d,S^{\lambda})\le n.\]
Moreover, suppose that $2\lambda\in X^+_r$ for some positive integer $r$. Then for all $0\le n\le 2p-4$, we have \[ \dim\opH^n(\Sigma_d,S^{\lambda})\le F(n+1)+(r-1)F(n).\]
\end{theorem} 

\subsection{Cohomology for the finite group $SL_2(p^s)$ and the algebraic group $SL_2$ on simple modules}
Let $G=SL_2, G(p^s)=SL_2(p^s)$ be the finite group of Lie type of $SL_2$ defined over the finite field $\mathbb{F}_{p^s}$, and assume that $p\ge 3$. The present section is devoted to applying our techniques to obtain explicit bounds for the cohomology dimension of the finite group $G(p^s)$ and of the algebraic group $G$ with coefficients in simple modules. We first compute an exponential upper bound for $\dim\opH^n(G(p^s), L)$ for all simple modules $L$, which requires a combination of the description of the $\Ext^n$ groups by Carlson in \cite{Ca:1983} and our combinatorial techniques. Then using the result of Cline, Parshall, Scott and van der Kallen on stability of generic cohomology \cite{CPSvdK} and a recursive formula of Parker in \cite{Pa:2006}, we establish the same upper bound for $SL_2$. Interestingly, this strategy is opposite of the common usage of generic cohomology, as one uses it to derive results in finite groups from those in algebraic groups.  

Recall that simple modules $L(m)$ of $G$ are parametrized by non-negative integers $m$ in $X^+=\mathbb{N}$. We define
\[ X^+_s=\{ m\in X^+: m< p^s \}\]
to be the set of {$p^s$-restricted weights} in $X^+$. Then every simple module over $G(p^s)$ is the restriction of $L(m)$ for some $m\in X^+_{s}$. Using Steinberg's Tensor Product Theorem, one can describe a simple module of $G(p^s)$ as follows
\[ L(m)\cong L(m_1)\otimes L(m_2)^{(1)}\otimes\cdots\otimes L(m_s)^{(s-1)} \]
where $m_i\in X^+_1$ for all $i$. For this reason, each simple module of $G(p^s)$ can be labeled by an $s$-tuple $\mathfrak m=(m_1,\ldots,m_s)\in (X^+_1)^s$. We denote it by $L_{\mathfrak m}$. It is easy to see that for each $\mathfrak m\in(X^+_1)^s$, we have $L_{\mathfrak m}=L(m)$ with $m=\sum_{i=1}^sm_ip^{i-1}\in X^+_s$. 
  
Suppose $L, M$ are simple $G(p^s)$ modules. The extension groups $\Ext^n_{kG(p^s)}(L,M)$ are completely described by Carlson in \cite{Ca:1983}. In particular, as a $k$-vector space bases of these extensions are indexed by the number of nonnegative integer solutions of certain inequalities and congruences. We use our calculations to estimate $\dim\Ext^n_{kG(p^s)}(L,M)$ as $L$ and $M$ vary. For convenience, we first rephrase the result of Carlson in terms of our notation.
   
\begin{theorem}\cite[Theorem 2.6]{Ca:1983}\label{Carlson}
Let $s$ be a positive integer. Suppose $\mathfrak d, \mathfrak f\in (X^+_1)^s$. For all $n\ge 0$, the dimension of $\Ext^n_{kG(p^s)}(L_{\mathfrak{d}},L_{\mathfrak f})$ as a $k$-vector space is the number of $s$-tuples $\mathfrak a, \mathfrak b, \mathfrak k\in\mathbb{N}^s$ satisfying the following conditions:
\begin{enumerate}
\item $2(a_1+\cdots+a_s)+b_1+\cdots+b_s=n,$
\item $b_i=0$ or $1$ for each $i$,
\item $a_i=b_i=0$ if $d_i$ or $f_i$ is $p-1$,
\item if $b_i=0$, then 
\[  \max\{0,f_i-d_i\}\le k_i\le 
\begin{cases}
f_i~~&\text{if}~~~a_i=0,\\
\min\{ f_i, p-d_i-2\}~~&\text{if}~~~a_i>0,
\end{cases}
\]
while if $b_i=1$ then \[ \max\{0, d_i+f_i+2-p\}\le k_i\le \min\{d_i, f_i\}, \]
\item $\displaystyle{2\left(p\sum_{i=1}^sa_ip^{i-1}+\sum_{i=1}^sb_ip^{i-1}\right)\equiv \sum_{i=1}^s(d_i-f_i+2k_i-2b_id_i)p^{i-1}~\quad (\text{mod}~p^s-1)}$.
\end{enumerate}
\end{theorem}

Now we prove the main result of this section.

\begin{theorem}\label{finite group}
For all $n\ge 0, s\ge 1$, and $\mathfrak d, \mathfrak f\in(X^+_1)^s$, we have
\[ \dim_k\Ext_{kG(p^s)}^n(L_{\mathfrak d},L_{\mathfrak f})\le (2n+2\max_i\{d_i\}+7)C_n\prod_{i=1}^s\left(\min\{d_i, f_i\}+1\right).\]  
Consequently, if $\mathfrak d$ is fixed, the dimension of $\Ext_{kG(p^s)}^n(L_{\mathfrak d},L_{\mathfrak f})$ has an exponential upper bound for all $\mathfrak f$. In particular, when $\mathfrak d=(0,\ldots,0)$, we have for all $s\ge 1$ and $\mathfrak f\in(X^+_1)^s$,
\[ \dim_k\opH^n(G(p^s),L_{\mathfrak f})\le (2n+7)C_n. \]
\end{theorem}

\begin{proof}
Firstly, Proposition \ref{exponential bound for V(m)} gives us that for each $\mathfrak k\in\mathbb{N}^s$ and $t\in\mathbb{Z}$, the number of $s$-tuples $\mathfrak a, \mathfrak b$ satisfying (1), (2), and a rewritten form of (5) 
\begin{align}\label{rewritten form of (5)} 
2\left(p\sum_{i=1}^sa_ip^{i-1}+\sum_{i=1}^sb_i(1+d_i)p^{i-1}\right)= \sum_{i=1}^s(d_i-f_i+2k_i)p^{i-1}+t(p^s-1)
\end{align}
is bounded by $C_n$. Now an upper bound for $\dim_k\Ext_{kG(p^s)}^n(L_{\mathfrak d},L_{\mathfrak f})$ can be established by counting the number of possible $\mathfrak k$ and $t$. From (4) in Theorem \ref{Carlson}, we see that the number of candidates for $\mathfrak k$ is at most $\prod_{i=1}^s\left(\min\{d_i, f_i\}+1\right)$. On the other hand, the left-hand side of \eqref{rewritten form of (5)} is at most $2(n+\max_i\{d_i\}+1)p^s$ while the right-hand side
\[ \sum_{i=1}^s(d_i-f_i+2k_i)p^{i-1} \ge -\sum_{i=1}^sf_ip^{i-1}\ge -(p^s-1) \] 
and 
\[ \sum_{i=1}^s(d_i-f_i+2k_i)p^{i-1}\le 3\sum_{i=1}^sd_ip^{i-1}\le 3(p^s-1). \] 
It follows that the number of possible $t$ is $2n+2\max_i\{d_i\}+7$, which is a constant depending only on $n$ and $\mathfrak d$. Finally, we have shown that the number of $s$-tuples $\mathfrak a, \mathfrak b, \mathfrak k$ satisfying $(1), (2), (4),$ and $(5)$ in Theorem \ref{Carlson} is at most $(2n+2\max_i\{d_i\}+7)C_n\prod_{i=1}^s\left(\min\{\alpha_i,\beta_i\}+1\right)$; hence we obtain the first inequality of the theorem. The other parts follow immediately by fixing $\mathfrak d$. 
\end{proof}

\begin{remark}
The upper bound for $\dim\opH^n(G(p^s),L)$ in the above theorem does not depend on $s$ and the simple module $L$. We call such number a {\it universal bound} for $SL_2(p^s)$. In general, it is shown in \cite{BNPPSS} that for a simple algebraic group $G$ of arbitrary rank, there exists a universal bound for $\dim\opH^n(G(p^s),L)$ depending only on $n$ and $\rank(G)$. However, determining such a bound is a challenging task.   
\end{remark}
Our result also implies the following
\begin{corollary}
The sequence $\displaystyle{\left(\max_{s\ge 1, m\in X^+_s}\{\dim\opH^n(G(p^s),L(m))\}\right)}$ grows exponentially.
\end{corollary}

\begin{proof} Let $\gamma_n$ denote the general term in the sequence.
By Theorem \ref{finite group}, it suffices to find an exponential lower bound of each $\gamma_n$. For sufficiently large $s$ and $s'$, a result on generic cohomology in \cite{CPSvdK} gives us
\[  \opH^n(SL_2(p^s),L(m))\cong\opH^n(SL_2, L(m)^{(s')}) \]
for all $n\ge 0$. Then the example of Stewart in \cite[Remark 2.8]{Ste} shows that the dimension of the right-hand side grows at least exponentially.
\end{proof}

Now we return to the cohomology of $G=SL_2$. This time we employ the previous results for finite groups and generic cohomology to obtain an exponential upper bound for the dimension of $\opH^n(G,L)$ for all simple $G$-modules $L$. More precisely, we have

\begin{theorem}\label{simple G-module}
For all $n\ge 0$ and $m\in X^+$, we have \[\dim\opH^n(G,L(m))\le (2n+7)C_n.\] In other words, the sequence $\displaystyle{\left(\max_{m\in X^+}\{ \dim\opH^n(G,L(m))\}\right)_n}$ has exponential growth.
\end{theorem}

\begin{proof}
For each $m\in X^+$, the generic cohomology in \cite{CPSvdK} and Theorem \ref{finite group} gives us
\begin{align*} 
\dim\opH^n(G,L(p^dm)) &=\dim\opH^n(G,L(m)^{(d)})\\
&=\dim\opH^n(G(p^s),L(m))\\
&\le (2n+7)C_n, 
\end{align*}
for each $n\ge 0$. Note that we consider here $d$ and $s$ sufficiently large, especially $s$ needs to be chosen so that $L(m)$ remains a simple module of $G(p^s)$. Next, using Theorem 4.3 in \cite{Pa:2006} we get
\[ \dim\opH^n(G,L(p^dm))=\dim\opH^n(G,L(p^{d-1}m))+\sum_{i=1}^{\lfloor\frac{n}{2}\rfloor}\dim\Ext^{n-2i}(V(2i),L(p^{d-1}m)). \]
It follows that $\dim\opH^n(G,L(p^{d-1}m))\le (2n+7)C_n$. Repeating the argument, we obtain $$\dim\opH^n(G,L(m))\le (2n+7)C_n.$$ Since $m$ was chosen arbitrarily, we finish our proof.
\end{proof}

\begin{remark}
Theorem \ref{simple G-module} is much stronger than the observation made in \cite{Ste} where it is shown that there exists a sequence of $m_n\in X^+$ such that the sequence $\displaystyle{\left(\dim\opH^n(G,L(m_n))\right)}$ grows exponentially. Note also that our method does not rely on Kazhdan--Lusztig polynomial coefficients as expected in \cite{PS:2011}. It would be great if one could extend our results to higher rank groups.

We end this paper by proposing a couple of interesting questions on universal bounds for higher rank groups. 
\end{remark}

\begin{question}
Suppose $G$ is a simple algebraic group defined over $k$. Do the dimensions of $\opH^n(G(p^s), L), \opH^n(G, L(\lambda))$, $\opH^n(G,V(\lambda))$ have exponential universal bounds, only depending on $n$, for $L, L(\lambda),$ and $V(\lambda)$ respectively simple modules of $G(p^s)$, $G$, and Weyl modules of $G$?
\end{question}

\section*{Acknowledgments}
The second author would like to thank Dan Nakano and Chris Bendel for useful suggestions.

\providecommand{\bysame}{\leavevmode\hbox to3em{\hrulefill}\thinspace}


\begin{thebibliography}{19}

\bibitem[1]{BNP1}
C. Bendel, D. Nakano, and C. Pillen, {\em Second cohomology for Frobenius kernels and related structures}, Adv. Math., \textbf{209} (2007), 162--197.   

\bibitem[2]{BNP2}
C. Bendel, D. Nakano, and C. Pillen, {\em Third cohomology for Frobenius kernels and related structures}, submitted, \href{http://arxiv.org/pdf/1410.2322.pdf}{http://arxiv.org/pdf/1410.2322.pdf}.

\bibitem[3]{BNPPSS}
C. Bendel, N. Nakano, B. Parshall, C. Pillen, L. Scott and D. Stewart, {\em Bounding extensions for finite groups and Frobenius kernels}, to appear in Algebras and Representation Theory.

\bibitem[4]{Ca:1983}
J.F~Carlson, {\em The cohomology of irreducible modules over $SL(2,p^n)$.}, Proc. London Math. Soc., \textbf{47} (1983), 480--492.
 
\bibitem[5]{CE}
A. Cox and K. Erdmann, {\em On $\Ext^2$ between Weyl modules for quantum $GL_n$}, Math. Proc. Cambridge Philos. Soc., \textbf{128} (2000), 441--463.

\bibitem[6]{CPS}
E. Cline, B. Parshall, and L. Scott, {\em Cohomology, hyperalgebras, and Representations}, J. Algebra, \textbf{63} (1980), 98--123.

\bibitem[7]{CPSvdK}
E.~Cline, B.~Parshall, L.~Scott, and W.~van der Kallen, {\em Rational and generic cohomology}, Invent. Math., \textbf{39} (1977), 143--163. 

\bibitem[8]{E:1995}
K. Erdmann, {\em $\Ext^1$ for Weyl modules for $SL_2(K)$}, Math. Z., \textbf{218} (1995), 447--459.

\bibitem[9]{EHP:2013}
K. Erdmann, K.C. Hannabuss, and A. Parker, {\em Bounding and unbounding higher extensions for $SL_2$}, J. Algebra, \textbf{389} (2013), 98--118.

\bibitem[10]{Gr:1984}
E. Grosswald, {\em Topics from the Theory of Numbers}, Springer Science \& Business Media, 1984.

\bibitem[11]{He:2009}
D.J.~Hemmer, {\em Cohomology and generic cohomology of Specht modules for the symmetric group}, J. Algebra, \textbf{322} (2009), 1498--1515.

\bibitem[12]{HN:2006}
D.J.~Hemmer and D.K.~Nakano, {\em On the cohomology of Specht modules}, J. Algebra, \textbf{306} (2006) 191--200.

\bibitem[13]{Hum}
J.~Humphreys, {\em Modular representations of finite groups of Lie type}, London Mathematical Society Lecture Note Series, \textbf{326}, Cambridge University Press, Cambridge, 2006. 

\bibitem[14]{Jan:2003}
J.C.~Jantzen, {\em Representations of algebraic groups}, Mathematical Surveys and Monographs, \textbf{107}, American Mathematical Society, Providence, RI, 2003.

\bibitem[15]{KN:2001}
A.S. Kleshchev and D.K. Nakano, {\em On comparing the cohomology of general linear and symmetric groups}, Pacific
J. Math. \textbf{201} (2001), 339--355.

\bibitem[16]{Ngo:2013}
N.V.~Ngo, {\em Cohomology for Frobenius kernels of $SL_2$}, J. Algebra, \textbf{396} (2013), 39--60.

\bibitem[17]{Pa:2006}
A.E.~Parker, {\em Higher extensions between modules for $SL_2$}, Adv. Math., \textbf{209} (2007), 381--405.

\bibitem[18]{Ste}
D.I.~Stewart, {\em Unbounding Ext.}, J. Algebra, \textbf{365} (2012), 1--11.


\bibitem[19]{PS:2011}
B.J.~Parshall and L.L~Scott, {\em Bounding Ext for modules for algebraic groups, finite groups and quantum groups}, Adv. Math., \textbf{226} (2011), 2065--2088. 

\end{thebibliography}
\end{document}